\documentclass[11pt,twoside]{amsart}
\usepackage{enumerate}

\newtheorem{thm}{Theorem}[section]
\newtheorem{prop}[thm]{Proposition}
\newtheorem{lem}[thm]{Lemma}
\newtheorem{cor}[thm]{Corollary}

\theoremstyle{definition}
\newtheorem{defin}[thm]{Definition}

\theoremstyle{remark}
\newtheorem{rem}[thm]{Remark}

\numberwithin{equation}{section}
\newcommand{\x}{\times}
\newcommand{\ox}{\otimes}
\newcommand{\ttt}{{\tilde t}}
\newcommand{\ensemble}[2]{\{\,#1\mid#2\,\}}
\newcommand{\vespa}{\vspace{1em}}
\newcommand{\C}{{\mathbb C}}
\newcommand{\Z}{{\mathbb Z}}
\newcommand{\CaD}{\mathcal D} 
\newcommand{\CE}{\mathcal E}       
\newcommand{\CH}{\mathcal H}
\newcommand{\CI}{\mathcal I}
\newcommand{\CM}{\mathcal M}  
\newcommand{\CN}{\mathcal N}
\newcommand{\CO}{\mathcal O}
\newcommand{\TX}{{T^*X}}
\newcommand{\TL}{{T^*\Lambda}}
\newcommand{\TYX}{{T^*_YX}}
\newcommand{\dTYX}{{\dot T^*_YX}}
\newcommand{\OX}{{{\mathcal O_X}}}
\newcommand{\cOXY}{\widehat{\mathcal O}_{X|Y}}
\newcommand{\DX}{{\mathcal D_X}}
\newcommand{\EX}{{{\mathcal E_X}}}
\newcommand{\EV}{{{\mathcal E_V}}}
\newcommand{\EIX}{{{\mathcal E^\infty_X}}}
\newcommand{\BYX}{{{\mathcal B_{Y|X}}}}
\newcommand{\BIYX}{{{\mathcal B^\infty_{Y|X}}}}
\newcommand{\CYX}{{{\mathcal C_{Y|X}}}}
\newcommand{\CIYX}{{{\mathcal C^\infty_{Y|X}}}}
\newcommand{\infi}{{\scriptstyle(\infty)}}
\newcommand{\infun}{{\scriptstyle(\infty,1)}}
\newcommand{\CHM}{{Ch(\CM)}}
\newcommand{\TCHM}{{\widetilde{Ch}(\CM)}}
\newcommand{\mCHM}[1]{{Ch_{\Lambda}#1(\CM)}}
\newcommand{\mCHN}[1]{{Ch_{\Lambda}#1(\CN)}}
\newcommand{\mTCHM}[1]{{\widetilde{Ch}_{\Lambda}#1(\CM)}}
\DeclareMathOperator{\Ext}{\mathcal Ext}
\DeclareMathOperator{\End}{\mathcal End}
\newcommand{\ga}{\alpha}
\newcommand{\gth}{\theta}          
\newcommand{\gTh}{\Theta}
\newcommand{\gvt}{\vartheta}
\newcommand{\gl}{\lambda}          
\newcommand{\gL}{\Lambda}
\newcommand{\gm}{\mu}
\newcommand{\gn}{\nu}
\newcommand{\gx}{\xi} 
\newcommand{\gp}{\pi} 
\newcommand{\gS}{\Sigma}
\newcommand{\gt}{\tau}
\newcommand{\gf}{\varphi} 
\newcommand{\gO}{\Omega}
\newcommand{\tgp}{{\pi_\Lambda}}              

\begin{document}

\title{Regularity and b-functions for D-modules}
\author{Yves Laurent}
\address{ Institut Fourier Math\'ematiques \\
    UMR 5582  CNRS/UJF\\
       BP 74\\
F-38402 St Martin d'H\`eres Cedex\\
  FRANCE }
\email{Yves.Laurent@ujf-grenoble.fr}
\maketitle

\section*{Abstract} 

A holonomic ${\mathcal D}$-module on a complex analytic manifold
admits always a $b$-function along any submanifold. If the module is regular, it
admits also a regular $b$-function, that is a $b$-function with 
a condition on the order of the lower terms of the equation.
There is a weaker condition of regularity: regularity along
a submanifold. We prove that a module which is regular along
a submanifold admits a regular $b$-function along this submanifold.

\thispagestyle{empty}
\newpage

\section*{Introduction}

Let $f$ be a holomorphic function on an open set $X$ of $\C^n$. A $b$-function for $f$
is a polynomial $b\in\C[s]$ such that there exists a differential operator $P$ on $U$ with parameter $s$
satisfying an equation:
\begin{equation} \label{equ1} 
b(s)f(x)^s=P(s,x,D_x)f(x)^{s+1} \tag{$1$}
\end{equation}
The generator of the ideal of the $b$-functions associated to $f$ is usually called the Bernstein-Sato
polynomial of $f$ (see\cite{BJORK} for details).

This definition has been extended by Kashiwara \cite{INV2} 
to holonomic $\CaD$-modules. Let $X$ be a complex manifold and $Y$ a smooth hypersurface. Let
$\DX$ be the sheaf of differential operators on $X$. Let $\CM$ be a holonomic $\DX$-module
and $u$ a section of $\CM$.
A $b$-function for $u$ along $Y$ is an equation  
\begin{equation}\label{equ2}
b(t D_t)u=tP(t,x,t D_t,D_x)u \tag{$2$}
\end{equation}
satisfied by $u$. Here $(t,x)$ are local coordinates of $X$ such that $t$ is an equation for $Y$. A similar definition
exits for submanifolds of $X$ of any codimension.

The $b$-function is called \textbf{regular} if $P$ may be chosen so that its order as a differential operator
is not greater than the degree of $b$.

It has been proved by Kashiwara \cite{INV2} that a holonomic $\CaD$-module admits a $b$-function
along any smooth hypersurface and by Kashiwara-Kawa{\"\i} \cite{KKHOUCH} that a \textsl{regular} holonomic 
$\CaD$-module admits a \textsl{regular} $b$-function along any smooth hypersurface. 

There is another notion of regularity, the regularity of a $\CaD$-module along
a submanifold. The definition will be given in definitions \ref{def:first} and \ref{def:defreg}. 
In the case of a hypersurface, this is equivalent to the fact that 
solutions in the formal completion transversally to the hypersurface are convergent.
It has been proved that a regular holonomic module is regular along any submanifold in
\cite{KKREG}. It is also a direct consequence of the definition that if a $\CaD$-module
admits a regular $b$-function along a submanifold, it is regular along it.
We will prove here that a module which is regular along a submanifold $Y$ admits a regular
$b$-function along $Y$.

The problem is better understood after microlocalization, that is for modules over the sheaf $\EX$
of microdifferential operators. Regularity and $b$-function may be defined for
a coherent $\EX$-module $\CM$ along a conic lagrangian submanifold $\gL$ of the cotangent bundle \cite{ENS}. 
All definitions are invariant under quantized canonical transformation. 
Concerning $\CaD$-modules, this will prove that the result is true not only for hypersurfaces 
but may be extended to any smooth subvariety of $X$. 

A regular holonomic $\EX$-module is regular along any smooth conic lagrangian subvariety of $\TX$. Conversely,
in the definition of Kashiwara-Kawa{\"\i} \cite{KKREG}, a holonomic $\EX$-module is regular if it is
regular along the smooth part of its characteristic variety.

In the first section of this paper, we briefly recall the definitions of $b$-functions and state the result for $\DX$-modules.
In the second section, we study the microlocal case. We define the filtrations and bifiltrations on holonomic $\EX$-modules 
In the third section, we study the equivalent definitions of regularity along a lagrangian manifold and we prove the existence of
a regular $b$-function as a product of $\gth^d$ by the  $b$-function $b(\gth)$.

\section{The case of differential equations}

In this first section, we briefly recall the definitions of $b$-functions and regularity in the framework
of $\DX$-modules.

\subsection{Filtrations and b-functions}\label{sec:vfil}

Let $X$ be a complex analytic manifold, $\CO_X$ the sheaf of holomorphic functions on $X$
and $\DX$ the sheaf of differential operators on $X$ with coefficients in $\OX$. Let
$\TX$ be the cotangent bundle to $X$ with canonical projection $\gp:\TX\to X$.
Let $Y$ be a submanifold of $X$.

The sheaf $\DX$ is provided with two filtrations. The first one is 
the filtration $(\CaD_{X,m})_{m\ge 0}$ by the usual order, that is the degree in the derivations. 
Here, we will denote this filtration by $F\DX$, i.e. $F_m\DX=\CaD_{X,m}$.
The corresponding graded ring $gr^F\DX$ is identified to $\gp_*\CO_{[\TX]}$
the sheaf of holomorphic functions on $\TX$ with polynomial coefficients 
in the fibers of $\gp$.

The second one is the V-filtration which has has been defined by Kashiwara in \cite{KVAN} as:
\begin{equation}
V_k\DX=\ensemble{P\in\DX|_Y}{\forall \ell\in \Z, P\CI_Y^\ell\subset \CI_Y^{\ell-k}}
\end{equation}
where $\CI_Y$ is the ideal of definition of $Y$ and $\CI_Y^\ell=\OX$ if $\ell\le 0$.

\vespa
Let $\gt:T_YX\to Y$ be the normal bundle to $Y$ in $X$ and $\CO_{[T_YX]}$ the sheaf of holomorphic
functions on $T_YX$ which are polynomial in the fibers of $\gt$. Let $\CO_{[T_YX]}[k]$ be the subsheaf
of $\CO_{[T_YX]}$ of homogeneous functions of degree $k$ in the fibers of $\gt$. There are canonical isomorphisms
between $\CI_Y^k/\CI_Y^{k-1}$ and $\gt_*\CO_{[T_YX]}[k]$, between $\bigoplus\CI_Y^k/\CI_Y^{k-1}$
and $\gt_*\CO_{[T_YX]}$. Hence 
the graded ring $gr^V\DX$ associated to the V-filtration on $\DX$ acts naturally 
on $\CO_{[T_YX]}$. An easy calculation \cite{SCHAPBOOK} shows that as a subring of $\End(\gt_*\CO_{[T_YX]})$ it is identified to 
$\gt_*\CaD_{[T_YX]}$ the sheaf of differential
operators on $T_YX$ with coefficients in $\CO_{[T_YX]}$ .

The Euler vector field $\gth$ of $T_YX$ is the vector field which acts on  $\CO_{[T_YX]}[k]$
by multiplication by $k$. Let $\gvt$ be any differential operator in $V_0\DX$ whose image in $gr^V_0\DX$
 is~$\gth$. Let $\CM$ be a coherent $\DX$-module and $u$ a section of $\CM$. 

\begin{defin}\label{def:bf}
A polynomial $b$ is a {\sl b-function}
for $u$ along $Y$ if there exists a differential operator $Q$ in $V_{-1}\DX$ such that $(b(\gvt)+Q)u=0$. 
\end{defin}

\begin{defin}\label{def:regbf}
A polynomial $b$ of degree $m$ is {\sl a regular b-function} for $u$ along $Y$ if there exists a differential operator $Q$ in $V_{-1}\DX\cap F_m\DX$ such that $(b(\gvt)+Q)u=0$. 
\end{defin}

The set of $b$-functions is an ideal of $\C[T]$, if it is not zero we call a generator of this
ideal \textsl{"the" b-function} of $u$ along $Y$.

Remark that the set of regular b-function for $u$ is not always an ideal of $\C[T]$.

\subsection{Regularity} 

A  holonomic module is regular its formal solutions converge at each point. 
More precisely, if $x$ is a point of $X$,
$\CO_{X,x}$ the ring of germs of $\OX$ at $x$ and $\mathfrak m$ the maximal ideal of $\CO_{X,x}$,
let us denote by $\widehat\CO_{X,x}$ 
the formal completion of $\CO_{X,x}$ for the $\mathfrak m$-topology. The holonomic $\DX$-module
$\CM$ is {\sl regular} if and only if:
$$\forall j\ge0,\quad \forall x\in X,\quad  \Ext_\DX^j(\CM,\widehat\CO_{X,x})= \Ext_\DX^j(\CM,\CO_{X,x})$$

This is the direct generalization of the definition in dimension $1$ (Ramis \cite{RAMISGAGA}). An algebraic and microlocal definition
is given by Kashiwara-Kawa{\"\i} in \cite{KKREG} and they prove that, for a $\DX$-module, it is
equivalent to this one.

A weaker notion is the regularity along a submanifold. This has been studied in several papers 
\cite{THESE},\cite{ENS},\cite{LME}. The definition will be given in section \ref{sub:defreg} 
using microcharacteristic varieties. Here we give a more elementary definition for $\DX$-modules
which works only for hypersurfaces.

Let $Y$ be a smooth hypersurface of $X$, $\gvt$ is the vector field of section \ref{sec:vfil} and 
$\CM$ be a holonomic $\DX$-module $\CM$ defined in a neighborhood of $Y$.
\begin{defin}\label{def:first}
The holonomic $\DX$-module $\CM$ is \textsl{regular along} $Y$ if any section $u$
of $\CM$ is annihilated by a differential operator of the form $\gvt^N + P + Q$
where $P$ is in $F_{N-1}\DX\cap V_{0}\DX$
and $Q$ is in $F_N\DX\cap V_{-1}\DX$
\end{defin}

Let us denote by $\cOXY$ the formal completion of $\OX$ along $Y$, that is
$$\cOXY=\varprojlim_k {\OX/\CI_Y^k}$$

We proved in \cite{INV} and \cite{LME} that $\CM$ is regular along $Y$ if and only if 
\begin{equation}
\label{equ:byx}
\forall j\ge0,\quad  \Ext_\DX^j(\CM,\cOXY)= \Ext_\DX^j(\CM,\OX)
\end{equation}

We can still state definition \ref{def:first} and equation (\ref{equ:byx}) when $Y$ 
is not a hypersurface but they are not equivalent. Correct statements will be given in section
\ref{sec:micro} in the microlocal framework.

\vespa
It has been proved by Kashiwara in \cite{INV2}, that
if $\CM$ is a holonomic $\DX$-module, there exists a $b$-function for any section $u$ of $\CM$ along
any submanifold $Y$ of $X$.
Kashiwara and Kawa{\"\i} proved in \cite{KKHOUCH} that if $\CM$ is a regular holonomic $\DX$-module,
there exists a regular $b$-function for any section $u$ of $\CM$ along
any submanifold $Y$ of $X$.

We will prove in section \ref{sec:main} that if $\CM$ is a holonomic $\DX$-module which is regular 
along a submanifold $Y$ then it admits a regular $b$-function along $Y$.

\subsection{Explicit formulas in local coordinates}

Consider local coordinates $(x_1,\dots,x_p,t)$ of $X$ such that $Y$ is the hypersurface
$Y=\ensemble{(x,t)\in X}{t=0}$. Then $T_YX$ has coordinates $(x_1,\dots,x_p,\ttt)$.

The Euler vector field  $\gth$ of $T_YX$ is $\gth=\ttt D_\ttt$ and we can choose 
$\gvt=t D_t$. Then a $b$-function is an equation
$$b(tD_t)+tQ(x,t,D_x,tD_t)$$
and it is a regular $b$-function if the order of $Q$ is less or equal to the degree of $b$.

The module $\CM$ is regular along $Y$ if any section of $\CM$ is annihilated by an equation
$$(tD_t)^N+P(x,D_x,tD_t)+tQ(x,t,D_x,tD_t)$$
with $P$ of order less or equal to $N-1$ and $Q$ of order less or equal to $N$. 

\section{Microdifferential equations}\label{sec:micro}

In this section, we review basic definitions of filtration, V-filtration and
bifiltrations on $\EX$-modules. 
Details may be founded in \cite{BJORK} or \cite{SCHAPBOOK}, and \cite{ENS}.

\subsection{V-filtration on microdifferential operators}\label{sec:ext}

We denote by $\EX$ the sheaf of microdifferential operators of \cite{SKK}. It is filtered by the order,
we will denote that filtration by $\EX = \bigcup F_k\EX$ and call it the usual filtration.

Let $\CO_{(\TX)}$ be the sheaf of holomorphic functions on $\TX$ which are finite sums of homogeneous functions
in the fibers of $\gp:\TX\to X$. The graded ring $gr^F\EX$ is isomorphic to $\gp_*\CO_{(\TX)}$ \cite{SCHAPBOOK}.

In \cite{ENS}, we extended the definitions of V-filtrations and $b$-functions
to microdifferential equations and lagrangian subvarieties of the cotangent bundle.

Let $\gL$ be a lagrangian conic submanifold of the cotangent bundle $\TX$ and $\CM_\gL$ be a holonomic 
$\EX$-module supported by $\gL$. A section $u$ of $\CM\gL$ is non degenerate if the ideal of $\CO_\TX$ generated by
the principal symbols of the microdifferential operators annihilating $u$ is the ideal of definition of $\gL$.
The module $\CM\gL$ is a simple holonomic $\EX$-module if it is generated by a non degenerate section $u_\gL$.
Such a module always exists locally \cite{SKK}.

Let $\CM_{\gL,k}=\CE_{X,k}u_\gL$. The V-filtration on $\EX$ along $\gL$ is defined by:
\begin{equation}
V_k\EX=\ensemble{P\in\EX|_\gL}{\forall \ell\in \Z, P\CM_{\gL,\ell}\subset \CM_{\gL,\ell+k}}
\end{equation}
This filtration is independent of the choices of $\CM_\gL$ and $u_\gL$, so it is globally defined \cite[Prop. 2.1.1.]{ENS}. 

Let $\CO_\gL[k]$ be the sheaf of holomorphic functions on $\gL$ homogeneous of degree $k$ in the
fibers of $\gL\to X$ and $\CO_{(\gL)}=\bigoplus_{k\in\Z}\CO_\gL[k]$. 
There is an isomorphism between $\CM_{\gL,k}/\CM_{\gL,k-1}$ and $\CO_{\gL,k}$. By this isomorphism
the graded ring $gr^V\EX$ acts on $\CO_{(\gL)}$ and may be identified to the sheaf $\CaD_{(\gL)}$ of 
differential operators on $\gL$ with coefficients in $\CO_{(\gL)}$.

All these definitions are invariant under quantized canonical transformations \cite{ENS}.

The restriction to the zero section $T^*_XX$ of $\TX$ of the sheaf $\EX$ is the sheaf
$\DX$ of differential operators. If $\gL$ is the conormal bundle $\TYX$ to a submanifold
$Y$ of $X$ then the V-filtration induced on $\DX$ by the V-filtration of $\EX$ is the same
than the V-filtration defined in section \ref{sec:vfil}. The associated graded ring is the
restriction to $Y$ of $\CaD_{(\gL)}$ which is the sheaf $\CaD_{[\gL]}$ of differential
operators with coefficients polynomial in the fibers of $\gL\to X$. 

The correspondence between the isomorphism $gr^V\DX\simeq \CaD_{[T_YX]}$ of section \ref{sec:vfil}
and the isomorphism  $gr^V\DX\simeq \CaD_{[\TYX]}$ is given
by the partial Fourier transform associated to the duality between the normal bundle $T_YX$ 
and the conormal bundle $\TYX$.

\begin{rem}\label{rem:fourier}
Consider local coordinates $(x_1,\dots,x_p,t_1,\dots,t_d) $ of $X$ and $Y=\ensemble{(x,t)\in X}{t=0}$.
Let $(x,\ttt)$ be the corresponding coordinates of $T_YX$ and $(x,\gt)$ the corresponding coordinates of 
$\TYX$. 
Let $\gth=\sum_{i=1}^d \ttt_iD_{\ttt_i}$ be the Euler vector field of the fiber bundle $T_YX$  and 
$\gth_\gL$ be the Euler vector field of $\gL=\TYX$. 

The Fourier transform is given by $\ttt_i\mapsto -D_{\gt_i}$ and $D_{\ttt_i}\mapsto \gt_i$. So $\gth$
is mapped to $-d-\gth_\gL$ where $d$ is the codimension of $Y$.
\end{rem}

\subsection{Filtrations on $\CE$-modules}\label{sec:vfil2}

Let $\CM$ be a (left) coherent $\EX$-module. A F-filtration (resp. a V-filtration) of $\CM$
is a filtration compatible with the F-filtration (resp. V-filtration) of $\EX$.  

A good filtration of $\CM$ is a filtration which is
locally of the form $\CM_k=\sum_{i=1}^N\CE_{X,k-k_i}u_i$ for $(u_1,\dots,u_N)$ local sections of $\CM$
and $(k_1,\dots,k_N)$ integers. In the same way, a good V-filtration of $\CM$ is
locally of the form $V_k\CM=\sum_{i=1}^N V_{k-k_i}\EX u_i$.

If $\CM$ is provided with a good F-filtration, the associated graded module
$gr^F\CM$ is a coherent module over $gr^F\EX=\gp_*\CO_{(\TX)}$. Then $\CO_{\TX}\ox_{\gp^{-1}gr^F\EX}\gp^{-1}gr^F\CM$ 
is a coherent $\CO_{\TX}$-module which defines a positive analytic cycle on $\TX$ independent of the good filtration.
This cycle is the \textsl{characteristic cycle} of $\CM$ and denoted by $\TCHM$. 
Its  support is the \textsl{characteristic variety} of $\CM$ denoted $\CHM$, it is equal to the support
 of the module $\CM$ itself.

The characteristic variety of a coherent $\EX$-module is a \textsl{homogeneous involutive} subvariety of $\TX$. 
When its dimension is minimal, that is when it is a lagrangian, the module is \textsl{holonomic}.

If $\CM$ is provided with a good V-filtration, the graded module
$gr^V\CM$ is a coherent module over $gr^V\EX=\CaD_{(\gL)}$ hence $gr^V\CM$ is a coherent
$\CaD_{(\gL)}$-module. The characteristic cycle of this module is a positive analytic cycle on 
$T^*\gL$ and its support is involutive. According to \cite{ENS} and \cite{LME}, this characteristic cycle
is called the
microcharacteristic cycle of $\CM$ of type $\infty$ and is denoted by $\mTCHM\infi$. The corresponding
microcharacteristic variety is $\mCHM\infi$. 
We have the following fundamental result:

\begin{thm}[Theorem 4.1.1 of \cite{ENS}]\label{thm:grholo}
The dimension of the characteristic variety of $gr^V\CM$ is less or equal to the dimension of 
the characteristic variety of $\CM$. In particular, $gr^V\CM$ is holonomic if $\CM$ is holonomic.
\end{thm}

\vespa
As $gr^V_k\EX$ is identified to the subsheaf
of $\CaD_{(\gL)}$ of differential operators $P$ satisfying $[\gth,P]=kP$, 
the Euler vector field $\gth$ acts on $gr^V\CM$, so we may define a morphism
$\gTh$ by $\gTh=\gth-k$ on $gr^V_k\CM$. As  $[\gth,P]=+kP$ on $gr^V_k\EX$, $\gTh$ commutes with the action
of  $gr^V\EX$ that is defines a section of $\End_{gr^V\EX}(gr^V\CM)$.

\begin{defin}\label{def:bfonct}
The set of polynomials $b$ satisfying 
$$b(\gTh)gr^V(\CM)=0$$
 is an ideal of $\C[T]$. When this ideal is not
zero, a generator is called \textsl{the $b$-function of $\CM$ for the $V$-filtration along $\gL$}.
\end{defin}

The $b$-function depends on the good $V$-filtration, its  roots are shifted by integers
by change of V-filtration.

The $b$-function is invariant under quantized canonical transform. 
Let $Y$ be a submanifold of $X$ of codimension $d$. If we want to identify the $b$-function of definition 
\ref{def:bfonct} with that of definition \ref{def:bf}, we have to replace $\gvt$ in definition \ref{def:bf}
by $-d-\gvt$ because of remark \ref{rem:fourier}.

\begin{thm}
If $\CM$ is a holonomic $\EX$-module, there exists a $b$-function for $\CM$ along
any lagrangian submanifold $\gL$ of $\TX$.
\end{thm}

This theorem has been proved first by Kashiwara for $\CaD$-modules in \cite{INV2}. It has been 
proved for $\EX$-module in \cite{ENS}
as a corollary of theorem \ref{thm:grholo}. In fact, as $gr^V\CM$ is holonomic,  $\End_{gr^V\EX}(gr^V\CM)$
is finite dimensional and the $b$-function is the minimal polynomial of $\gTh$ on $gr^V\CM$.

\subsection{Bifiltration}

From the two filtrations on $\EX$, we get a bifiltration:
\begin{equation}
\forall (k,j)\in\Z\x\Z\qquad W_{k,j}\EX=F_j\EX\cap V_k\EX
\end{equation}

To this bifiltration is associated the bigraded ring:
\begin{equation}\label{def:bif}
gr^W\EX = \bigoplus_{(k,j)} gr^W_{k,j}\EX\quad\textrm{with}\quad gr^W_{k,j}\EX=W_{k,j}\EX/(W_{k-1,j}\EX+W_{k,j-1}\EX)
\end{equation}

This bigraded ring is equal to the graded ring of $gr^V\EX=\CaD_{(\gL)}$ (this latter with the standard filtration).
So it is isomorphic to $\tgp_*\CO_{[\TL]}$, the sheaf of holomorphic functions on $\TL$ 
which are polynomial in the fibers of $\tgp:\TL\to \gL$ and sum of homogeneous functions for the second action
of $\C^*$ on $\TL$. "Second action" means action induced on $\TL$ by the action of $\C^*$ on $\gL$.

Let $\CM$ be a (left) coherent $\EX$-module. A good bifiltration of $\CM$ is a bifiltration compatible 
with the bifiltration of $\EX$ which is locally generated by local sections $(u_1,\dots,u_N)$  of $\CM$ This means
that there are integers $(\gm_1,\dots,\gm_N)$ and $(\gn_1,\dots,\gn_N)$ such that 
$W_{k,j}\CM=\sum_{\gn=1}^NW_{k-k_\gn,j-j_\gn}\EX u_\gn$.

The bigraded module associated to the bifiltration is defined by 
\begin{equation}\label{def:bifm}
gr^W\CM = \bigoplus_{(k,j)} gr^W_{k,j}\CM\quad\textrm{with}\quad gr^W_{k,j}\CM=W_{k,j}\CM/(W_{k-1,j}\CM+W_{k,j-1}\CM)
\end{equation}

\begin{prop}\label{prop:subbigrad}
\ 
\begin{enumerate}
\item The bigraded module associated to a good bifiltration of a coherent $\EX$-module is a coherent
 $gr^W\EX=\tgp_*\CO_{[\TL]}$ module.
\item The bifiltration induced on a coherent submodule by a good bifiltration is still a good bifiltration. 
\end{enumerate}
\end{prop}

This result may be found in \cite{TMF} or \cite{ENS}, let us briefly recall its proof.

We denote by $\EV$ the sheaf $gr^V_0\EX$ provided with the filtration induced by that of $\EX$ (cf \cite{KKREG}).

\begin{lem}\label{lem:equiv}
Outside of the zero section of $\gL$, there is an equivalence between the following data:
\begin{itemize}
\item a good bifiltration of $\CM$
\item a coherent sub-$\EV$-module of $\CM$ which generates $\CM$ and is provided with a good filtration.  
\end{itemize}
\end{lem}

\begin{proof}
Outside of its zero section, we may transform $\gL$ into the conormal
to a hypersurface by a canonical transformation, that is in local coordinates:
$$\gL=\ensemble{(x,t,\gx,\gt)\in T^*X}{t=0, \gx=0, \gt\ne 0}$$
Then the microdifferential operator $D_t$ is invertible. It is an element of $W_{1,1}\EX$.

Consider a good bifiltration $W\CM$ of $\CM$ which is locally generated by sections $u_1,\dots,u_N$  of $\CM$.
Multiplying them by suitable powers of $D_t$ we may assume that $(u_1,\dots,u_N)$ belong to
$$\CN=\bigcup_{j\in\Z}W_{0,j}\CM$$
Then $\CN$ is the sub-$\EV$-module of $\CM$ generated by $(u_1,\dots,u_N)$ and, by \cite[Prop. 1.1.10.]{KKREG},
$\CN$ is an $\EV$-coherent submodule of $\CM$. The sections $(u_1,\dots,u_N)$ define a good filtration of 
$\CN$ and $\CM=\EX\otimes\CN$.

Conversely, let $\CN$ be coherent sub-$\EV$-module of $\CM$ which generates $\CM$ and is provided with a good filtration.  
The good filtration is given locally by sections $(u_1,\dots,u_N)$ which define a good bifiltration of $\CM$. 
\end{proof}

\begin{proof}[Proof of proposition \ref{prop:subbigrad}]
From \cite[\S 2.6.]{THESE} (see also \cite{TMF}), we know that he filtration of $\EV$ is "a good noetherian filtration"
that is a zariskian filtration in the notation of \cite{SCHAPBOOK}. Then the results of \cite{SCHAPBOOK} show that
good $\EV$-filtrations induce good  $\EV$-filtrations on submodules and that the graded ring associated to
a good $\EV$-filtration is a coherent $gr\EV$-module. The result may be then deduce from the previous lemma out
of the zero section of $\gL$.

We get the result on the zero section of $\gL$ by adding a dummy variable as usual (cf \cite[\S A1]{KKREG}).
\end{proof}

Let $\CO_{\TL}$ be the sheaf of holomorphic functions on $\TL$.
As $gr_W\CM$ is a $\tgp_*\CO_{[\TL]}$ coherent module, $\CO_\TL\otimes_{\CO_{[\TL]}}\tgp^{-1}gr_\CM$ is a 
coherent $\CO_\TL$-module. Its support is independent of the good bifiltration.
It is called the microcharacteristic variety of type $(\infty,1)$ and is denoted by $\mCHM\infun$. 
This microcharacteristic variety is an involutive bihomogeneous subvariety of $\TL$ \cite{THESE},\cite{ENS}. 

Let us remark that the same microcharacteristic variety has been defined by Teresa Monteiro Fernandes in \cite{TMF},
where it is denoted by $C¹_\gL(\CM)$, and in \cite{THESE} as the support of the tensor product of the module by
a sheaf of 2-microdifferential operators.

\subsection{Local coordinates}

Assume that the conic lagrangian manifold $\gL$ is the conormal to a submanifold $Y$ of $X$. (By a canonical transformation
we may always transform $\gL$ into a conormal).

Let $(x_1,\dots,x_p,t_1,\dots t_d)$ be local coordinates of $X$ such that $Y=\ensemble{(x,t)\in X}{t=0}$. Then
$\gL=\ensemble{(x,t, \gx,\gt)\in \TX}{t=0,\gx=0}$ and $\TL$ has coordinates $(x,\gt,x^*,\gt^*)$.

The Euler vector field of $\gL$ is $\gth=\sum_{i=1}^N \gt_i D_{\gt_i}$ and its principal symbol is the function
$\gf=\sum_{i=1}^N \gt_i \gt^*_i = <\gt,\gt^*>$. The hypersurface $S_\gL$ of $\TL$ defined by the equation $\gf$
is a canonical hypersurface which will be considered in next section.

For the $W$ bifiltration, the operator $x_i$, $i=1,\dots,p$, is of bi-order $(0,0)$, $t_j$, $j=1,\dots,d$ is of bi-order $(-1,0)$, 
$D_{x_i}$ is of bi-order $(0,1)$, $D_{t_j}$ is of bi-order $(1,1)$.

The operator $\gvt$ may be chosen here as $\sum t_j D_{t_j}$, it is of bi-order $(0,1)$.

\section{Regularity along a lagrangian conic submanifold}\label{sub:defreg}

\subsection{Equivalent definitions of regularity}
The Euler vector field $\gth_\gL$ of $\gL$ is a differential operator on $\gL$, its characteristic
variety is a canonical subvariety of  $\TL$ which will be denoted by $S_\gL$. As in section 
\ref{sec:ext}, $\CM_\gL$ is a simple holonomic  $\EX$-module supported by $\gL$.

\begin{defin}\label{def:defreg}
Let $\gL$ be a lagrangian conic submanifold of $\TX$ and $\CM$ be a holonomic $\EX$-module. The module $\CM$
is {\bf regular along $\gL$} if and only if it satisfies the following equivalent properties:

\begin{enumerate}[i)]\itemsep 2pt
\item The microcharacteristic variety $\mCHM\infun$ is contained in $S_\gL$.
\item The microcharacteristic variety $\mCHM\infun$ is lagrangian.
\item The microcharacteristic variety $\mCHM\infun$ is equal to $\mCHM\infi$ 
\item $\forall j\ge0,\quad  \Ext_\EX^j(\CM,\CM_\gL)= \Ext_\EX^j(\CM,\EIX\ox_\EX\CM_\gL)$
\end{enumerate}
\end{defin}

We recall that $\EIX$ is the sheaf of microdifferential operators of infinite order.

If $\gL$ is the conormal bundle to a submanifold $Y$ of $X$, we may take $\CM_\gL=\CYX$ the
sheaf of holomorphic microfunctions of \cite{SKK}. Then (iv) is reformulated as:
 
(iv)' $\forall j\ge0,\quad  \Ext_\EX^j(\CM,\CYX)= \Ext_\EX^j(\CM,\CIYX)$

\vespa
Let us now briefly remain how the equivalence between the items of definition \ref{def:defreg}
is proved.

By theorem  \ref{thm:grholo}, if $\CM$ is holonomic, $gr^V\CM$ is holonomic hence its 
characteristic variety $\mCHM\infi$ is lagrangian. So if  $\mCHM\infun$ is equal to $\mCHM\infi$
it is lagrangian that is (iii)$\Rightarrow$(ii).

The microcharacteristic variety $\mCHM\infun$ is defined by the coherent $gr_B\CM$-module $\tgp_*\CO_{[\TL]}$
hence is bihomogeneous, that is homogeneous in the fibers of $\TL$ and homogeneous for the action of $\C^*$
induced by the action of $\C^*$ on $\gL$. By section 4.3. of \cite{ENS}, any lagrangian bihomogeneous submanifold
of $\TL$ is contained in $S_\gL$. In fact, if $\gS$ is lagrangian the canonical 2-form $\gO$ of $\TL$ vanishes on $\gS$
and if it is bihomogeneous the vector fields $v_1$ and $v_2$ associated to the two actions of $\C^*$ are tangent to $\gS$.
Hence $\gO(v_1,v_2)$ vanishes on $\gS$ and an easy calculation of \cite{ENS} shows that $\gO(v_1,v_2)$ is an equation of $S_\gL$.
So (ii)$\Rightarrow$ (i). 

Finally, the main parts of the result are (i) $\Rightarrow$(iv) which is given by corollary 4.4.2. of \cite{IRR} and (iv) $\Rightarrow$(iii) given by theorem 2.4.2 of \cite{LME}.

\subsection{Regular $b$-function}\label{sec:main}

Let $\gL$ be a conic lagrangian submanifold of the cotangent bundle $T^*X$. As before, $\gth_\gL$ is the Euler vector field
of $\gL$ and $\gvt_\gL$ is a microdifferential operator in $V_0\EX$ whose image in $gr_0^V\EX$ is $\gth_\gL$.

\begin{defin}\label{def:microregb}
Let $\CM$ be a coherent $\EX$-module with a good bifiltration $W\CM$ relative to $\gL$. A {\sl regular b-function} for
$W\CM$ is a polynomial $b$ such that:
$$\forall (k,j)\in \Z\x\Z, \quad b(\gvt_\gL-k)W_{k,j}\CM \subset W_{k-1,j+n}\CM$$
where $n$ is the degree of $b$.

if $u$ is a section of $\CM$, a  {\sl regular b-function} for $u$ is a regular  $b$-function for the 
canonical bifiltration of the module $\CN = \EX u$ that is $W_{k,j}\CN = (W{k,j}\EX) u$.
\end{defin}

Remark that if $W\CM$ is a good bifiltration, we get a good V-filtration by setting:
$$V_k\CM=\bigcup_{j\in \Z}W_{k,j}\CM$$
So, considering non regular $b$-function, we may define a $b$-function for $W$ to be a $b$-function
for the corresponding $V$-filtration.

\begin{prop}\label{prop:regul}
Let $\CM$ be a holonomic $\EX$-module which a good bifiltration admitting a regular $b$-function,
then $\CM$ is regular along $\gL$.
\end{prop}

Indeed, the existence of a regular $b$-function implies immediately the point (i) of definition \ref{def:defreg}.

By the following theorem and its corollary, the converse is true which gives another condition equivalent to the regularity
along $\gL$.

\begin{thm}\label{thm:mainE}
Let $\CM$ be holonomic $\EX$-module regular along a the conic lagrangian manifold $\gL$.
Any section $u$ of $\CM$ admits locally a {\sl regular b-function}.

If $b(\gth)$ is the $b$-function of $u$, there exits an integer $d$ such that
$\gth^d b(\gth)$ is a {\sl regular b-function} for $u$.
\end{thm}

\begin{proof}
Let $\gth$ be the Euler vector field of $\gL$ and $\gvt$ an operator of $\EX$ whose image in
$gr^V\EX$ is $\gth$. By definition, the principal symbol $\gf$ of $\gth$ is an equation of the canonical
hypersurface $S_\gL$ of $\TL$.

If $\CM$ is regular along $\gL$, the submodule $\CN=\EX u$ is regular along $\gL$. We consider on $\CN$ the canonical bifiltration,
that is $W_{kj}\CN=(W_{kj}\EX) u$.
By definition \ref{def:defreg} of regularity,  $\mCHN\infun$ is contained in $S_\gL$
hence $gr^W\CM$ is annihilated by a power of $\gf$. As the operator $\gvt$ is of bi-order $(0,1)$, there is an integer $q$ such that
$$\gvt^q W_{k,j}\CN\subset W_{k-1,j+q}\CN + W_{k,j+q-1}\CN$$
Hence there is an operator $P\in W_{-1,q}\EX$ and  an operator $Q\in W_{0,q-1}\EX$ such that  
$$\gvt^q u = Pu+Qu$$

We have $(\gvt^q - P)^2 u = Q(\gvt^q - P)u + [\gvt^q - P,Q]u = Q^2 u + [\gvt^q - P,Q]u$. The commutator 
 $[\gvt^q - P,Q]$ belongs to $W_{0,2(q-1)}\EX$ as well as $Q^2$ so $(\gvt^q - P)^2 u = Q_2 u$ with $Q_2\in W_{0,2(q-1)}\EX$.
 By induction we find that $(\gvt^q - P)^r u = Q_r u$ with $Q_r\in W_{0,r(q-1)}\EX$.

\vespa
Let $b$ be the b-function of $u$. If $p$ is the degree of $b$, there is some integer $r$
such that $b(\gvt) W_{0,j}\CN \subset W_{-1,j+p+r}\CN$.

We get 
$$b(\gvt)(\gvt^q - A)^r u \subset b(\gvt)W_{0,rq-r}\CN \subset W_{-1,p +rq}\CN$$

The operator $(\gvt^q - A)^r$ is equal to $\gvt^{rq} - A_r$ where $A_r$ belongs to $W_{-1,rq}\EX$.
hence 
$b(\gvt)\gvt^{rq} u \subset W_{-1,p +rq}\CN$ and 
as $p+rq$ is the degree of $b(\gvt)\gvt^{rq}$, this is a regular $b$-function for $u$.
\end{proof}

\begin{rem}\label{rem:mainE}
It is clear in the proof that $\gth^d b(\gth)$ may be replaced by $(\gth-\ga)^d b(\gth)$
for any complex number $\ga$.
\end{rem}

\begin{cor}\label{cor:mainE}
Let $\CM$ be holonomic $\EX$-module regular along the conic lagrangian manifold $\gL$.
Any  good bifiltration $W\CM$ relative to $\gL$ admits locally a {\sl regular b-function}.

If $b(\gth)$ is the $b$-function associated to this bifiltration, there exits an integer $d$ such that
$\gth^d b(\gth)$ is a {\sl regular b-function}.
\end{cor}

\begin{proof}
Let $(u_1,\dots,u_N)$ be local generators of the bifiltration. By definition, there are integers $(\gl_1,\dots,\gl_N)$
and $(\gn_1,\dots,\gn_N)$ such that 
$$W_{kj}\CM = \sum_{i=1}^NW_{k-\gl_i,j-\gn_i}\EX u_i$$
If $b(\gth)$ is the $b$-function of the bifiltration, $b(\gth-\gl_i)$ is a $b$-function for the section $u_i$.
By theorem \ref{thm:mainE} and remark  \ref{rem:mainE}, there exists an integer $d_i$ such that $(\gth-\gl_i)^{d_i} b(\gth-\gl_i)$
is a regular $b$-function for $d_i$. Let $d$ be the maximum of $d_i$ for $i=1,\dots,N$.

As remarked in section \ref{sec:vfil2},  $[\gth,P]=+kP$ on $gr^V_k\EX$ hence if $P$ is an operator of 
$W_{k-\gl_i,j-\gn_i}\EX$ we have $\gvt P = P (\gvt+k-\gl_i) +R$ with $R$ in $W_{k-\gl_i-1,j-\gn_i}\EX$.

If $u$ is a section of $W_{kj}\CM$ we have $u=\sum P_i u_i$ with $P_i$ in $W_{k-\gl_i,j-\gn_i}\EX$ and 
$(\gth-k)^d b(\gth-k)$ is a regular $b$-function for $u$.
\end{proof}

\subsection{Application to $\DX$-modules}

Previous results have been established for $\EX$-modules on the whole of $\TX$ including the zero section
that is also for $\DX$-modules.

If $x$ is a point of the zero section $X$ of $\TX$, a conic lagrangian submanifold of $\TX$
defined in a neighborhood of $x$ is the conormal $\TYX$ to a submanifold $Y$ of $X$.
Let $\BIYX=\CH^1_Y(\OX)$ be the cohomology of $\OX$ and 
$\BYX=\CH^1_{[Y]}(\OX)$ the corresponding algebraic cohomology.

If $Y$ has codimension $1$ we have:
$$\CIYX/\CYX=\gp^{-1}(\BIYX/\BYX)\qquad \textrm{with}\quad \gp:\TYX\to Y$$
So if $\CM$ is a holonomic $\DX$-module condition (iv) of definition \ref{def:defreg} is equivalent to

(iv)'' $\forall j\ge0,\quad  \Ext_\DX^j(\CM,\BYX)= \Ext_\DX^j(\CM,\BIYX)$

and by \cite{LME} this is equivalent to formula \ref{equ:byx}.

If $Y$ has codimension greater than $1$, the situation is different. By a canonical transformation,
the situation is microlocally equivalent to the situation of codimension $1$. But in the definition
\ref{def:defreg} we have to consider points of $\dTYX=\TYX-Y$ the complementary of the zero section in the conormal bundle. 
Regularity is then local on $\dTYX$. The equivalence between algebraic and analytic definitions that is between conditions
(i) to (iii) in definition \ref{def:defreg} and condition (iv) have to be stated locally on $\dTYX$

Let $\gp:\dTYX\to Y$ and $U$ an open set of $Y$. If a $\DX$-module is regular along $\dTYX$ globally on  $\gp^{-1}U$
 we have 
$$\forall j\ge0,\quad  \Ext_\DX^j(\CM,\BYX)|_U= \Ext_\DX^j(\CM,\BIYX)|_U$$
but the converse is not true.

\subsection{ A counterexample}

We give here a very simple example due to T. Oaku which shows that the regular $b$-function may be
different from the $b$-function.

Let $X= \C^2$ and $Y$ the hypersurface $Y=\ensemble{(x,t)\in X}{t=0}$. Consider the $\DX$-module
$\CM=\DX/\CI$ where $\CI$ is the ideal of $\DX$ generated by the two operators $P=D_t^2$ and
$Q=D_t+D_x^2$.

This module is nonzero as the function $1$ is a solution. Its characteristic variety is the zero section
of the cotangent bundle $T^*X$ hence it is isomorphic as a $\DX$-module to a power of the sheaf $\OX$ of
holomorphic functions on $X$. So, this module is regular holonomic.

Let $u$ be the class of the operator $1$ in $\CM$. As $u$ is annihilated by $tQ=tD_t+tD_x^2$, the $b$-function
of $u$ along $Y$ is $\gvt=tD_t$. But this not a regular $b$-function and a regular $b$-function for $u$ along
$Y$ is given by $t^2P=t^2D_t^2=tD_t(tD_t-1)=\gvt(\gvt-1)$.

Remark however that $\CM$ is isomorphic to $(\OX)^4$. The function $1$ in $\OX$ is annihilated by $D_t$ hence its
$b$-function is $\gvt=tD_t$ and is a regular $b$-function. So $\CM$ is generated by four sections for which
the regular $b$-function is equal to the $b$-function $\gvt$, that is there is on $\CM$ a bifiltration for which
the $b$-function is a regular $b$-function.

We can see also on this example that the analytic cycle defined by the graded ring of a good bifiltration depend
on the good bifiltration. This situation is very different from the case of filtrations where good filtrations
define additive objects as analytic cycles which are independent of the good filtration \cite[Prop. 1.3.1]{SCHAPBOOK}.

\providecommand{\bysame}{\leavevmode\hbox to3em{\hrulefill}\thinspace}
\providecommand{\MR}{\relax\ifhmode\unskip\space\fi MR }
\providecommand{\MRhref}[2]{%
  \href{http://www.ams.org/mathscinet-getitem?mr=#1}{#2}
}
\providecommand{\href}[2]{#2}

\vespa\vespa

\enddocument

\end